\documentclass[9pt,oneside,english]{amsart}
\usepackage[T1]{fontenc}
\usepackage[latin9]{inputenc}
\usepackage[a4paper]{geometry}
\usepackage{amsthm}
\usepackage{amssymb}

\makeatletter
\numberwithin{equation}{section}
\numberwithin{figure}{section}
\theoremstyle{plain}
\newtheorem{thm}{\protect\theoremname}
  \theoremstyle{definition}
  \newtheorem{defn}[thm]{\protect\definitionname}
  \theoremstyle{plain}
  \newtheorem{lem}[thm]{\protect\lemmaname}
  \theoremstyle{remark}
  \newtheorem{rem}[thm]{\protect\remarkname}
  \theoremstyle{plain}
  \newtheorem{prop}[thm]{\protect\propositionname}

\usepackage{xypic}
\usepackage{pstricks,pst-node}
\usepackage[cmtip,all]{xy}
\theoremstyle{definition}
\newtheorem{parn}{}[section]

\makeatother

\usepackage{babel}
  \providecommand{\definitionname}{Definition}
  \providecommand{\lemmaname}{Lemma}
  \providecommand{\propositionname}{Proposition}
  \providecommand{\remarkname}{Remark}
\providecommand{\theoremname}{Theorem}

\begin{document}
\author{Adrien Dubouloz} 
\address{Adrien Dubouloz\\ CNRS\\ Institut de Math\'{e}matiques de Bourgogne\\ Universit\'{e} de Bourgogne\\ 9 Avenue Alain Savary\\ BP 47870\\ 21078 Dijon Cedex\\ France} \email{Adrien.Dubouloz@u-bourgogne.fr} 
\thanks{Research supported in part by ANR Grant "BirPol" ANR-11-JS01-004-01.}

\title{Flexible bundles over rigid affine surfaces}
\begin{abstract}
We construct a smooth rational affine surface $S$ with finite automorphism
group but with the property that the group of automorphisms of the
cylinder $S\times\mathbb{A}^{2}$ acts infinitely transitively on
the complement of a closed subset of codimension at least two. Such
a surface $S$ is in particular rigid but not stably rigid with respect
to the Makar-Limanov invariant. 
\end{abstract}
\maketitle

\section*{Introduction }

A complex affine variety $X$ is called \emph{rigid} if it does not
admit non trivial algebraic actions of the additive group $\mathbb{G}_{a}=\mathbb{G}_{a,\mathbb{C}}$.
This is the case for ``most'' affine varieties, for instance for
every affine curve different from the affine line $\mathbb{A}^{1}$
and for every affine variety whose normalization has non negative
logarithmic Kodaira dimension. The notion was actually introduced
by Crachiola and Makar-Limanov in \cite{CML05} under the more algebraic
equivalent formulation that the \emph{Makar-Limanov invariant} $\mathrm{ML}(X)$
of $X$, which is defined as the algebra consisting of regular functions
on $X$ invariant under all algebraic $\mathbb{G}_{a}$-actions, is
equal to the coordinate ring $\Gamma(X,\mathcal{O}_{X})$ of $X$. 

Among many important questions concerning this invariant, the understanding
of its behavior under the operation consisting of taking cylinders
$X\times\mathbb{A}^{n}$, $n\geq1$, over a given affine variety $X$
has focused a lot of attention during the last decade, in connexion
with the Zariski Cancellation Problem. Of course, rigidity is lost
even when passing to the cylinder $X\times\mathbb{A}^{1}$ since these
admit non trivial $\mathbb{G}_{a}$-actions by translations on the
second factor. But one could expect that such actions are essentially
the unique possible ones in the sense that the projection $\mathrm{pr}_{X}:X\times\mathbb{A}^{1}\rightarrow X$
is invariant for every $\mathbb{G}_{a}$-action on $X\times\mathbb{A}^{1}$,
a property which translates algebraically to the fact that $\mathrm{ML}(X\times\mathbb{A}^{1})=\mathrm{ML}(X)$.
This property was indeed established by Makar-Limanov \cite{MLNotes}
and this led to wonder more generally whether a rigid variety is \emph{stably
rigid} in the sense that the equality $\mathrm{ML}(X\times\mathbb{A}^{n})=\mathrm{ML}(X)$
holds for arbitrary $n\geq1$. Stable rigidity of smooth affine curves
is easily confirmed as a consequence of the fact that a smooth rigid
curve does not admit any dominant morphism from the affine line, and
more generally every rigid affine curve is in fact stably rigid \cite{CML05}.
Stable rigidity is also known to hold for smooth factorial rigid surfaces
by virtue of a result of Crachiola \cite{CraThesis}, and without
any indication of a potential counter-example, it seems that the implicit
working conjecture so far has been that every rigid affine variety
should be stably rigid. 

In this article, we construct a smooth rigid surface $S$ which fails
stable rigidity very badly, the cylinder $S\times\mathbb{A}^{2}$
being essentially as remote as possible from a rigid variety in terms
of richness of $\mathbb{G}_{a}$-actions on it. Here ``richness''
has to be interpreted in the sense of a slight weakening of the notion
of flexibility introduced recently in \cite{AFKKZ12,AKZ12} that we
call \emph{flexibility in codimension one}: a normal affine variety
$X$ is said to be flexible in codimension one if for every closed
point $x$ outside a possibly empty closed subset of codimension two
in $X$, the tangent space $T_{x}X$ of $X$ at $x$ is spanned by
tangent vectors to orbits of $\mathbb{G}_{a}$-actions on $X$. Clearly,
the Makar-Limanov invariant of a variety with this property is trivial,
consisting of constant functions only. Now our main result can be
stated as follows: 
\begin{thm}
\label{thm:Main} Let $V\subset\mathbb{P}^{3}$ be smooth cubic surface
and let $D=V\cap H$ be a hyperplane section of $V$ consisting of
the union of a smooth conic and its tangent line. Then $S=V\setminus D$
is a smooth rigid affine surface whose cylinder $S\times\mathbb{A}^{2}$
is flexible in codimension one.
\end{thm}
\noindent A noteworthy by-product is that while the automorphism
group $\mathrm{Aut}(S)$ of $S$ is finite, actually isomorphic to
$\mathbb{Z}/2\mathbb{Z}$ if the cubic surface $V$ is chosen general,
Theorem 0.1 in \cite{AFKKZ12} implies that $\mathrm{Aut}(S\times\mathbb{A}^{2})$
acts infinitely transitively on the complement of a closed subset
of codimension at least two in $S\times\mathbb{A}^{2}$. \\

Our construction is inspired by earlier work of Bandman and Makar-Limanov
\cite{BML08} which actually already contained the basic ingredients
to construct a counter-example to stable rigidity, in the form of
a lifting lemma for $\mathbb{G}_{a}$-actions which asserts that if
$\mathrm{q}:Z\rightarrow Y$ is a line bundle over a normal affine
variety $Y$ then $\mathrm{ML}(Z)\subseteq\mathrm{ML}(Y)$, and an
example of a non trivial line bundle $\mathrm{p}:L\rightarrow\tilde{S}$
over a smooth rational rigid affine surface $\tilde{S}$ for which
$\mathrm{ML}(L)\subsetneqq\mathrm{ML}(\tilde{S})$. Indeed, with these
informations, the property that $\mathrm{ML}(\tilde{S}\times\mathbb{A}^{2})$
is a proper sub-algebra of $\mathrm{ML}(\tilde{S})$ could have been
already deduced as follows: letting $\mathrm{p}':L'\rightarrow\tilde{S}$
be a line bundle representing the class of the inverse of $L$ in
the Picard group of $\tilde{S}$, the lifting lemma applied to the
rank $2$ vector bundle $E=L\oplus L'=L\times_{\tilde{S}}L'\rightarrow\tilde{S}$
considered as a line bundle over $L$ via the first projection implies
that $\mathrm{ML}(E)\subseteq\mathrm{ML}(L)\subsetneqq\mathrm{ML}(\tilde{S})$.
But combined with a result of Pavaman Murthy \cite{Mur69} which asserts
in particular that every vector bundle on such a surface $\tilde{S}$
is isomorphic to the direct sum of its determinant and a trivial bundle,
the construction of $E$ guarantees that it is isomorphic to the trivial
bundle $\tilde{S}\times\mathbb{A}^{2}$ and hence that $\mathrm{ML}(\tilde{S}\times\mathbb{A}^{2})\subsetneqq\mathrm{ML}(\tilde{S})$.
\\

Noting that the aforementioned of Pavaman Murthy also applies to surfaces
$S$ as in Theorem \ref{thm:Main}, the same construction can be used
for its proof provided that such an $S$ admits a line bundle $\mathrm{p}:L\rightarrow S$
whose total space is flexible in codimension one, and that flexibility
in codimension one lifts to total spaces of line bundles. The lifting
property follows easily from the fact that every line bundle admits
$\mathbb{G}_{a}$-linearizations, but the existence of a line bundle
$\mathrm{p}:L\rightarrow S$ with the desired property is trickier
to establish. To construct such a bundle, we exploit the fact that
$S$ admits an $\mathbb{A}^{1}$-fibration $\pi:S\rightarrow\mathbb{P}^{1}$,
i.e. a faithfully flat morphism with generic fiber isomorphic to affine
line. The strategy then consists in constructing a suitable $\mathbb{A}^{1}$-fibered
affine surface $\pi_{F}:S_{F}\rightarrow\mathbb{P}^{1}$ flexible
in codimension one and to which a variant of the famous Danielewski
fiber product trick \cite{Dan89} can be applied to derive the existence
of an affine threefold flexible in codimension one and carrying simultaneously
the structure of a line bundle over $S$ and $S_{F}$. \\

The article is organized as follows. In the first section we review
basic results about rigid and flexible affine varieties, with a particular
focus on the case of affine surfaces, and we establish that flexibility
in codimension one does indeed lift to total spaces of line bundles
(see Lemma \ref{lem:Flex-lift}). Section two is devoted to the study
of the class of affine surfaces $S$ considered in Theorem \ref{thm:Main}
and the construction of their aforementioned flexible mates $S_{F}$.
The appropriate variant of the Danielewski fiber product trick needed
to achieve the proof of Theorem \ref{thm:Main} is discussed in the
last section.

\section{Preliminaries on (stable) rigidity and flexibility }

\subsection{Rigid and flexible affine varieties}

\indent\newline\noindent Given a normal complex affine variety $X=\mathrm{Spec}(A)$,
we denote by $\mathcal{D}er_{\mathbb{C}}(\mathcal{O}_{X})\simeq\mathcal{H}om_{X}(\Omega_{X/\mathbb{C}}^{1},\mathcal{O}_{X})$
the sheaf of germs of $\mathbb{C}$-derivations from $\mathcal{O}_{X}$
to itself. It is a coherent sheaf of $\mathcal{O}_{X}$-modules whose
global sections coincide with elements of the $A$-module $\mathrm{Der}_{\mathbb{C}}(A)$
of $\mathbb{C}$-derivations of $A$. We denote by $\mathrm{LND}_{\mathbb{C}}(A)$
the sub-$A$-module of $\mathrm{Der}_{\mathbb{C}}(A)$ generated by
locally nilpotent $\mathbb{C}$-derivations, i.e. $\mathbb{C}$-derivations
$\partial:A\rightarrow A$ for which every element of $A$ is annihilated
by a suitable power of $\partial$. Recall that such derivations coincide
precisely with velocity vector fields of $\mathbb{G}_{a}$-actions
on $X$ (see e.g. \cite{FreuBook}). 
\begin{defn}
A normal affine variety $X=\mathrm{Spec}(A)$ is called:

a) Rigid if $\mathrm{LND}_{\mathbb{C}}(A)=\left\{ 0\right\} $, equivalently
$X$ does not admit non trivial $\mathbb{G}_{a}$-actions, 

b) Flexible in codimension $1$, or $1$-flexible for short, if the
support of the co-kernel of the natural homomorphism $\mathrm{LND}_{\mathbb{C}}(A)\otimes_{A}\mathcal{O}_{X}\rightarrow\mathcal{D}er_{X}(\mathcal{O}_{X})$
has codimension at least $2$ in $X$. 
\end{defn}
\begin{parn} The above definition of $1$-flexibility says equivalently
that there exists a closed subset $Z\subset X$ of codimension at
least $2$ such that the restriction of $\mathcal{D}er_{\mathbb{C}}(\mathcal{O}_{X})$
over $X\setminus Z$ is generated by elements of $\mathrm{LND}_{\mathbb{C}}(A)$.
A closed point $x\in X$ at which the natural homomorphism $\mathrm{LND}_{\mathbb{C}}(A)\otimes_{A}\mathcal{O}_{X,x}\rightarrow\mathcal{D}er_{X}(\mathcal{O}_{X})_{x}$
is surjective is called a flexible point of $X$, this property being
equivalent by virtue of Nakayama Lemma to the fact that the Zariski
tangent space $T_{x}X$ of $X$ at $x$ is spanned by the tangent
vectors to orbits of $\mathbb{G}_{a}$-actions on $X$. The set $X_{\mathrm{flex}}$
of flexible points is contained in the regular locus $X_{\mathrm{reg}}$
of $X$ and is invariant under the action of the automorphism group
$\mathrm{Aut}(X)$ of $X$. In particular, if there exists a flexible
point $x\in X$ such that the complement of the $\mathrm{Aut}(X)$-orbit
of $x$ is contained in a closed subset of codimension at least two,
then $X$ is flexible in codimension $1$. 

\end{parn}

\begin{parn} We warn the reader that our definition of flexibility
for a normal affine variety $X$ is weaker than the one introduced
earlier in \cite{AFKKZ12,AKZ12} which asks in addition that $X_{\mathrm{flex}}=X_{\mathrm{reg}}$.
Since for a $1$-flexible variety the set $X\setminus X_{\mathrm{flex}}$
has codimension at least two in $X$, this makes essentially no difference
for global properties of $X$ depending on regular functions, for
instance the Makar-Limanov invariant of a $1$-flexible affine variety
is trivial. Furthermore, all the properties of the regular locus of
a flexible variety in the sense of \emph{loc. cit.} hold for the open
subset $X_{\mathrm{flex}}$ of a $1$-flexible variety $X$, for instance
the sub-group of $\mathrm{Aut}(X)$ generated by its one-parameter
unipotent sub-groups acts infinitely transitively on $X_{\mathrm{flex}}$.

\end{parn}

\noindent Clearly, the only $1$-flexible affine curve is the affine
line $\mathbb{A}^{1}$. While the classification of flexible affine
surfaces in the stronger sense of \cite{AFKKZ12,AKZ12} is not known
and most probably quite intricate, $1$-flexible surfaces coincide
with the so-called \emph{Gizatullin surfaces} \cite{Giz71} with non
constant invertible functions. More precisely, we have the following
characterization (see also \cite[Example 2.3]{AFKKZ12}). 
\begin{thm}
\label{thm:Giz-Flex} For a normal affine surface $S$, the following
are equivalent:

a) $S$ is $1$-flexible, 

b) $S$ admits two $\mathbb{A}^{1}$-fibrations over $\mathbb{A}^{1}$
with distinct general fibers, 

c) $\Gamma(S,\mathcal{O}_{S}^{*})=\mathbb{C}^{*}$ and $S$ admits
a normal projective completion $S\hookrightarrow V$ whose boundary
is a chain of proper smooth rational curves supported on the regular
locus of $V$. \end{thm}
\begin{proof}
It is well known that every $\mathbb{A}^{1}$-fibration $q:S\rightarrow C$
over a smooth affine curve $C$ arises as the algebraic quotient morphism
$q:S\rightarrow S/\negmedspace/\mathbb{G}_{a}=\mathrm{Spec}(\Gamma(S,\mathcal{O}_{S})^{\mathbb{G}_{a}})$
of a non trivial $\mathbb{G}_{a}$-action on $S$. In particular,
the general fibers of such fibrations coincide with the general orbits
of a $\mathbb{G}_{a}$-action on $S$. Since a flexible surface admits
at least two $\mathbb{G}_{a}$-actions with distinct general orbits,
this provides two $\mathbb{A}^{1}$-fibrations on $S$ with distinct
general fibers and whose respective base curves are isomorphic to
$\mathbb{A}^{1}$ due to the fact that they are dominated by a general
fiber of the other fibration. Conversely, let $q_{i}:S\rightarrow\mathbb{A}^{1}$,
$i=1,2$, be $\mathbb{A}^{1}$-fibrations on $S$ associated with
a pair of $\mathbb{G}_{a}$-actions $\sigma_{1}$ and $\sigma_{2}$
on $S$ with distinct general orbits. Since the morphism $q_{1}\times q_{2}:S\rightarrow\mathbb{A}^{2}$
is quasi-finite \cite[Lemma 2.21]{Dub04}, it follows on the one hand
that general orbits of $\sigma_{1}$ and $\sigma_{2}$ intersect each
other transversally and on the other hand that the intersection $S_{0}$
of the fixed point loci of $\sigma_{1}$ and $\sigma_{2}$ is finite.
This implies in turn that every point in $S\setminus S_{0}$ can be
mapped by an element of the sub-group of $\mathrm{Aut}(S)$ generated
by $\sigma_{1}$ and $\sigma_{2}$ to a point $p\in S$ at which a
general orbit of $\sigma_{1}$ intersects a general orbit of $\sigma_{2}$
transversally. Such a point $p$ is certainly flexible. Therefore
every point outside the finite closed subset $S_{0}$ is a flexible
point of $S$ which proves the equivalence between a) and b). For
the equivalence b)$\Leftrightarrow$c) we refer the reader to \cite{Dub04}
(in which the statement of Theorem 2.4 should actually be corrected
to read: A normal affine surface with no non constant invertible functions
is completable by a zigzag if and only if it admits two $\mathbb{A}^{1}$-fibrations
whose general fibers do not coincide).
\end{proof}
\goodbreak

\subsection{Stable rigidity/stable flexibility.}

\subsubsection{Rigidity property for line bundles }

\begin{parn} \label{par:Ga-Base-lift} The total space of a line
bundle $\mathrm{p}:L\rightarrow X$ over an affine affine variety
$X=\mathrm{Spec}(A)$ always admits $\mathbb{G}_{a}$-actions by generic
translations along the fibers of $\mathrm{p}$, associated with locally
nilpotent $A$-derivations of $\Gamma(L,\mathcal{O}_{L})$. More precisely,
these derivations corresponds to $\mathbb{G}_{a,X}$-actions on $L$,
i.e. $\mathbb{G}_{a}$-actions on $L$ by $X$-automorphisms, and
are in one-to-one correspondence with global sections $s\in H^{0}(X,L)$
of $L$. Indeed, letting $\mathrm{p}:L=\mathrm{Spec}(\mathrm{Sym}(M^{\vee}))\rightarrow X$,
where $M\simeq H^{0}(X,L)$ is a locally free $A$-module of rank
$1$, one has $\Omega_{\mathrm{Sym}(M^{\vee})/A}^{1}\simeq\mathrm{Sym}(M^{\vee})\otimes_{A}M^{\vee}$
and the isomorphism 
\[
\mathrm{Der}_{A}(\mathrm{Sym}(M^{\vee}))\simeq\mathrm{Hom}_{\mathrm{Sym}(M^{\vee})}(\Omega_{\mathrm{Sym}(M^{\vee})/A}^{1},\mathrm{Sym}(M^{\vee}))\simeq\mathrm{Sym}(M^{\vee})\otimes_{A}M
\]
identifies $A$-derivations of $\Gamma(L,\mathcal{O}_{L})\simeq\mathrm{Sym}(M^{\vee})$
with global sections of the pull-back $\mathrm{p}^{*}L$ of $L$ to
its total space. Since a $\mathbb{G}_{a,X}$-action on $L$ corresponding
to a locally nilpotent $A$-derivation $\partial$ of $\mathrm{Sym}(M^{\vee})$
restricts on every fiber of $\mathrm{p}:L\rightarrow X$ to a $\mathbb{G}_{a}$-action
which is either trivial or a translation, it follows that the corresponding
section of $\mathrm{p}^{*}L$ is constant along the fibers of $\mathrm{p}:L\rightarrow X$
whence is the pull-back by $\mathrm{p}$ of a certain section $s_{\partial}\in H^{0}(X,L)$.
Consersely, every global section $s\in H^{0}(X,L)$ gives rise to
a $\mathbb{G}_{a,X}$-action on $L$ defined by $\sigma_{s}(t,\ell)=\ell+ts(\mathrm{p}(\ell))$
where the fiberwise addition and multiplication are given by the vector
space structure. More formally, viewing $\mathrm{p}:L\rightarrow X$
as a locally constant group scheme for the law $\mu:L\times_{X}L\rightarrow L$
induced by the addition of germs of sections, global sections $s\in H^{0}(X,L)$
give rise to homomorphisms $s:\mathbb{G}_{a,X}\rightarrow L$ of group
schemes over $X$ whence to $\mathbb{G}_{a,X}$-actions $\sigma_{s}=\mu\circ(s\times\mathrm{id}_{L}):\mathbb{G}_{a,X}\times_{X}L\rightarrow L$
on $L$.

\end{parn}

\begin{parn} Even though they are no longer rigid, it is natural
to wonder whether total spaces of line bundles $\mathrm{p}:L\rightarrow X$
over rigid varieties $X$ stay ``as rigid as possible'' in the sense
that they do not admit any $\mathbb{G}_{a}$-actions besides the $\mathbb{G}_{a,X}$-actions
described above. For the trivial line bundle $\mathrm{pr}_{X}:X\times\mathbb{A}^{1}\rightarrow X$,
the question was settled affirmatively by Makar-Limanov \cite{MLNotes}
(see also \cite[Proposition 9.23]{FreuBook}). Let us briefly recall
the argument for the convenience of the reader: viewing $\Gamma(X\times\mathbb{A}^{1},\mathcal{O}_{X\times\mathbb{A}^{1}})=A[x]=\bigoplus_{i\geq0}A\cdot x^{i}$
as a graded $A$-algebra, every nonzero locally nilpotent derivation
$\partial$ of $A[x]$ associated with a non trivial $\mathbb{G}_{a}$-action
on $X\times\mathbb{A}^{1}$ decomposes into a finite sum $\partial=\sum_{i\in\mathbb{Z}}\partial_{i}$
of nonzero homogeneous derivations $\partial_{i}:A[x]\rightarrow A[x]$
of degree $i\in\mathbb{Z}$, the top homogeneous component $\partial_{m}$
being itself locally nilpotent. Note that $m\geq-1$ for a nonzero
derivation and that derivations of the form $a\partial_{x}$ for a
certain $a\in A\setminus\{0\}$ correspond to the case $m=-1$. On
the other hand, if $m\geq0$ then $\partial_{m}=x^{m}\tilde{\partial}_{0}$
for a certain derivation of degree $0$ and since $\partial_{m}(x)\in x^{m+1}A\subset xA$,
$x$ must belong to the kernel of $\partial_{m}$. This implies that
$\tilde{\partial}_{0}$ is a locally nilpotent derivation of degree
$0$ whose restriction to $A=A\cdot x^{0}\subset A[x]$ is trivial
as $X$ is rigid. But since since $x\in\mathrm{Ker}(\tilde{\partial}_{0})=\mathrm{Ker}(\partial_{m})$,
$\tilde{\partial}_{0}$ whence $\partial$ would be the zero derivation,
a contradiction.

\end{parn}

\begin{parn} In contrast, as mentioned in the introduction, it was
discovered by Bandman and Makar-Limanov \cite{BML08} that the above
property can fail for non trivial line bundles. The fact that the
rigid surfaces considered in Theorem \ref{thm:Main} admit line bundles
$\mathrm{p}:L\rightarrow S$ with $1$-flexible total spaces (see
\S \ref{par:end_of_proof} below) shows that such total spaces can
be in general very far from being rigid. 

\end{parn}

\subsubsection{Lifting flexibility in codimension one to split vector bundles}

\begin{parn} The total space of the trivial line bundle $\mathrm{pr}_{X}:X\times\mathbb{A}^{1}\rightarrow X$
over a $1$-flexible (resp. flexible in the sense of \cite{AKZ12})
affine variety $X=\mathrm{Spec}(A)$ is again $1$-flexible (resp.
flexible). Indeed, every locally nilpotent derivation $\partial$
of $A$ canonically extends to a locally nilpotent derivation $\tilde{\partial}$
of $A[x]$ containing $x$ in its kernel in such way that the projection
$\mathrm{pr}_{X}:X\times\mathbb{A}^{1}\rightarrow X$ is equivariant
for the corresponding $\mathbb{G}_{a}$-actions on $X$ and $X\times\mathbb{A}^{1}$
respectively. It follows that for every point $p\in X\times\mathbb{A}^{1}$
dominating a flexible point $x$ of $X$, say for which $\mathcal{D}er_{\mathbb{C}}(\mathcal{O}_{X})_{x}$
is generated by the images of locally nilpotent derivations $\partial_{1},\ldots,\partial_{r}$
of $A$, the $\mathcal{O}_{X\times\mathbb{A}^{1},p}$-module $\mathcal{D}er_{\mathbb{C}}(\mathcal{O}_{X\times\mathbb{A}^{1}})_{p}$
is generated by the images of $\tilde{\partial}_{1},\ldots,\tilde{\partial}_{r}$
together with the image of the locally nilpotent $A$-derivation $\partial_{x}$
of $A[x]$. This implies that $\mathrm{pr}_{X}^{-1}(X_{\mathrm{flex}})\subset(X\times\mathbb{A}^{1})_{\mathrm{flex}}$
and hence that the set of non flexible points in $X\times\mathbb{A}^{1}$
has codimension at least two. Furthermore, $(X\times\mathbb{A}^{1})_{\mathrm{flex}}$
coincides with $(X\times\mathbb{A}^{1})_{\mathrm{reg}}$ in the case
where $X_{\mathrm{flex}}=X_{\mathrm{reg}}$. 

\end{parn}

\begin{parn} Even though different results related with lifts of
$\mathbb{G}_{a}$-actions on an affine variety $X$ to $\mathbb{G}_{a}$-actions
on total spaces of line bundles $\mathrm{p}:L\rightarrow X$ over
it exist in the literature (in particular, \cite[Lemma 9]{BML08}
and \cite[Corollary 4.5]{AFKKZ12}), it seems that the question whether
$1$-flexibility or flexibility of $X$ lifts to total spaces of arbitrary
line bundles over it has not been clearly settled yet. This is fixed
by the following cost free generalization: 

\end{parn}
\begin{lem}
\label{lem:Flex-lift} Let $X$ be a normal affine variety and let
$\mathrm{p}:E\rightarrow X$ be a vector bundle which splits as a
direct sum of line bundles. If $X$ is $1$-flexible (resp. flexible)
then so is the total space of $E$. \end{lem}
\begin{proof}
Since $E$ is isomorphic to the fiber product $L_{1}\times_{X}L_{2}\cdots\times_{X}L_{r}$
of line bundles $\mathrm{p}_{i}:L_{i}\rightarrow X$, we are reduced
by induction to the case of a line bundle $\mathrm{p}:L\rightarrow Y$
over a $1$-flexible (resp. flexible) affine variety. Recall that
for a connected algebraic group $G$ acting on a normal variety $Y$,
there exists an exact sequence of groups 
\[
0\rightarrow H_{\mathrm{alg}}^{1}(G,\Gamma(Y,\mathcal{O}_{Y}^{*}))\rightarrow\mathrm{Pic}^{G}(Y)\stackrel{\alpha}{\rightarrow}\mathrm{Pic}(Y)\rightarrow\mathrm{Pic}(G)
\]
where $\mathrm{Pic}^{G}(Y)$ denotes the group of $G$-linearized
line bundles on $Y$ and where $H_{\mathrm{alg}}^{1}(G,\Gamma(Y,\mathcal{O}_{Y}^{*}))$
parametrizes isomorphy classes of $G$-linearizations of the trivial
line bundle over $Y$ (see e.g. \cite[Chap. 7]{DolgBook}). In the
case where $G=\mathbb{G}_{a}$, this immediately implies that every
line bundle $\mathrm{p}:L\rightarrow Y$ admits a $\mathbb{G}_{a}$-linearization
(note furthermore that such a linearization is unique up to isomorphism
provided that $\Gamma(Y,\mathcal{O}_{Y}^{*})=\mathbb{C}^{*}$). 

It follows in particular that every $\mathbb{G}_{a}$-action on $Y$
can be lifted to a $\mathbb{G}_{a}$-action on $L$ preserving the
zero section $Y_{0}\subset L$ and for which the structure morphism
$\mathrm{p}:L\rightarrow Y$ is $\mathbb{G}_{a}$-invariant. So the
$1$-flexibility (resp. the flexibility) of $L$ follows from that
of $Y$ thanks to \cite[Corollary 4.5]{AFKKZ12}. But let us provide
a self-contained argument: the above property translates algebraically
to the fact that every locally nilpotent derivation $\partial$ of
$\Gamma(Y,\mathcal{O}_{Y})$ extends to a locally nilpotent derivation
$\tilde{\partial}$ of $\Gamma(L,\mathcal{O}_{L})$ mapping the ideal
$I_{Y_{0}}$ of $Y_{0}$ into itself and such that the induced derivation
on $\Gamma(Y_{0},\mathcal{O}_{Y_{0}})=\Gamma(L,\mathcal{O}_{L})/I_{Y_{0}}$
coincides with $\partial$ via the isomorphism $\Gamma(Y,\mathcal{O}_{Y})\stackrel{\sim}{\rightarrow}\Gamma(Y_{0},\mathcal{O}_{Y_{0}})$
induced by the restriction of $\mathrm{p}$. Since $Y$ is affine,
given any point $\ell\in L$, we can find a global section $s\in H^{0}(Y,L)$
which does not vanish at $y=\mathrm{p}(\ell)$. Now if $y$ is a flexible
point of $Y$, say for which $\mathcal{D}er_{\mathbb{C}}(\mathcal{O}_{Y})_{y}$
is generated by the images of locally nilpotent derivations $\partial_{1},\ldots,\partial_{r}$
of $\Gamma(Y,\mathcal{O}_{Y})$ then $\ell_{0}=\mathrm{p}^{-1}(y)\cap Y_{0}$
is a flexible point of $L$ at which $\mathcal{D}er_{\mathbb{C}}(\mathcal{O}_{L})_{\ell_{0}}$
is generated by the lifts $\tilde{\partial}_{1},\ldots,\tilde{\partial}_{r}$
of $\partial_{1},\ldots,\partial_{r}$ together with the locally nilpotent
derivation $\partial_{s}$ of $\Gamma(L,\mathcal{O}_{L})$ corresponding
to the $\mathbb{G}_{a,Y}$-action $\sigma_{s}:\mathbb{G}_{a,Y}\times_{Y}L\rightarrow L$
associated with $s$ (see \S \ref{par:Ga-Base-lift} above). Furthermore,
since $s$ does vanish at $y$, the $\mathbb{G}_{a}$-action induced
by $\sigma_{s}$ on $\mathrm{p}^{-1}(y)$ is transitive, and so $\mathrm{p}^{-1}(y)$
consists of flexible points of $L$. This shows that $\mathrm{p}^{-1}(Y_{\mathrm{flex}})\subset L_{\mathrm{flex}}$
and completes the proof. 
\end{proof}

\section{Construction of rigid and $1$-flexible $\mathbb{A}^{1}$-fibered
surfaces over over $\mathbb{P}^{1}$ }

In this section, we first consider affine surfaces $S_{R}$ which
arise as complements of well-chosen hyperplane sections of a smooth
cubic surface in $\mathbb{P}^{3}$. We check that they are rigid by
computing their automorphism groups and we exhibit certain $\mathbb{A}^{1}$-fibrations
$\pi_{R}:S_{R}\rightarrow\mathbb{P}^{1}$ on them. We then construct
auxiliary $1$-flexible $\mathbb{A}^{1}$-fibered surfaces $\pi_{F}:S_{F}\rightarrow\mathbb{P}^{1}$
which will be used later on in section three for the proof of Theorem
\ref{thm:Main}.

\subsection{A family of rigid affine cubic surfaces }

Most of the material of this sub-section is borrowed from \cite{DubKish12}
to which we refer the reader for the details. 

\begin{parn} Given a smooth cubic surface $V\subset\mathbb{P}^{3}$
and a line $L$ on it, the restriction to $V$ of the linear pencil
$\mathcal{H}_{L}=\left|\mathcal{O}_{\mathbb{P}^{3}}(1)\otimes\mathcal{I}_{L}\right|$
on $\mathbb{P}^{3}$ generated by planes containing $L$ can be decomposed
as $\mathcal{H}_{L}\mid_{V}=\mathcal{L}+L$ where $\mathcal{L}$ is
a base point free pencil defining a conic bundle $\Phi_{\mathcal{L}}:V\rightarrow\mathbb{P}^{1}$
with five degenerate fibers each consisting of the union of two lines.
The restriction $\Phi_{\mathcal{L}}\mid_{L}:L\rightarrow\mathbb{P}^{1}$
is a double cover and for every branch point $x\in\mathbb{P}^{1}$
of $\Phi_{\mathcal{L}}\mid_{L}$, the intersection of $V$ with the
corresponding hyperplane $H_{x}\in\mathcal{H}$ consists either of
a smooth conic tangent to $L$ or two distinct lines intersecting
$L$ in a same point, which is then an Eckardt point of $V$. The
second case does not occur if $V$ is chosen general. We fix from
now on a cubic surface $V$, a line $L$ on it and a hyperplane section
$D=H\cap V$ for which $D=L+C$ where $C$ is a smooth conic tangent
to $L$ at a point $p\in L$.

\end{parn}

\begin{parn} \label{par:Sr-Fib} Given a pair $\left(V,D\right)$
where $D=L+C$ as above, the surface $S_{R}=V\setminus D$ is affine
as $D$ is a hyperplane section of $V$. It comes equipped with an
$\mathbb{A}^{1}$-fibration $\pi_{R}:S_{R}\rightarrow\mathbb{P}^{1}$
which is obtained as follows: we let $\mu:V\rightarrow\mathbb{P}^{2}$
be the birational morphism obtained by contracting a $6$-tuple of
disjoint lines $L,F_{1},\ldots,F_{5}\subset V$ with the property
that each $F_{i}$, $i=1,\ldots,5$, intersects $C$ transversally.
Since $L$ is tangent to $C$, the image $\mu_{*}(C)$ of $C$ in
$\mathbb{P}^{2}$ is a cuspidal cubic. The rational pencil on $\mathbb{P}^{2}$
generated by $\mu_{*}(C)$ and three times its tangent $T$ at its
unique singular point $\mu(p)$ lifts to a rational pencil $\overline{q}:V\dashrightarrow\mathbb{P}^{1}$
having the divisors $C+\sum_{i=1}^{5}F_{i}$ and $3T+L$ as singular
members. Letting $\tau:\hat{V}\rightarrow V$ be a minimal resolution
of $\overline{q}$, the induced morphism $\overline{q}\circ\tau:\hat{V}\rightarrow\mathbb{P}^{1}$
is a $\mathbb{P}^{1}$-fibration whose restriction to $S_{R}=V\setminus D\simeq\hat{V}\setminus\tau^{-1}D$
is an $\mathbb{A}^{1}$-fibration $\pi_{R}:S_{R}\rightarrow\mathbb{P}^{1}$
with two degenerate fibers: one is irreducible of multiplicity three
consisting of the intersection of the proper transform of $T$ with
$S_{R}$ and the other is reduced, consisting of the disjoint union
of the curves $F_{i}\cap S_{R}\simeq\mathbb{A}^{1}$, $i=1,\ldots,5$
(see Figure \ref{fig:FibSR}). 

\begin{figure}[ht]
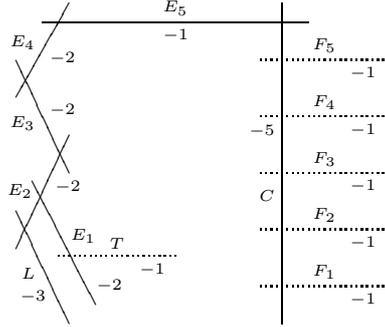
 
$$\xygraph{!{<0cm,0cm>;<0.7cm,0cm>:<0cm,0.5cm>::}
!{(2,-2)}="C1" !{(2,6.5)}="C2"  
!{(1.6,-1)}="F11" !{(1.6,0.5)}="F21" !{(1.6,2)}="F31" !{(1.6,3.5)}="F41" !{(1.6,5)}="F51"
!{(4,-1)}="F12" !{(4,0.5)}="F22" !{(4,2)}="F32" !{(4,3.5)}="F42" !{(4,5)}="F52"
!{(2.5,6)}="S1"!{(-2.5,6)}="S2"
!{(-2,6.5)}="E51" !{(-3,4)}="E52"
!{(-3,5)}="E41" !{(-2,2)}="E42"
!{(-2,3)}="E31" !{(-3,0)}="E32"
!{(-2.7,1.8)}="E21" !{(-1.5,-1.5)}="E22"
!{(-3,1)}="E11" !{(-2,-2)}="E12"
!{(-2.2,-0.2)}="T1" !{(0,-0.2)}="T2"
"F11"-@{.}^{F_{1}}_>(0.8){-1}"F12"
"F21"-@{.}^{F_{2}}_>(0.8){-1}"F22"
"F31"-@{.}^{F_{3}}_>(0.8){-1}"F32"
"F41"-@{.}^{F_{4}}_>(0.8){-1}"F42"
"F51"-@{.}^{F_{5}}_>(0.8){-1}"F52"
"C1"-^>(0.4){C}^>(0.6){-5}"C2"
"S1"-_{E_5}^{-1}"S2"
"E51"-_{E_4}^{-2}"E52"
"E41"-_{E_3}^{-2}"E42"
"E31"-_>(0.55){E_2}^(0.4){-2}"E32"
"E21"-^{E_1}^>(0.9){-2}"E22"
"E11"-_{L}_>(0.7){-3}"E12"
"T1"-@{.}^{T}_>(0.8){-1}"T2"
}
$$
\caption{The total transform of $\mu_*(C)\cup T \subset \mathbb{P}^2$  in the minimal resolution $\tau:\hat{V} \rightarrow V$ of $\overline{q}$. The plain curves correspond to the irreducible components of $\tau^{-1}(D)$ and the exceptional divisors $E_i$ of $\tau$ are numbered according to the order they are extracted.} 
\label{fig:FibSR} 
\end{figure} 

\end{parn}
\begin{rem}
\label{Rk:Alt-SrFib} Choosing an alternative $6$-tuple of disjoint
lines $F_{0,1}F_{0,2},F_{\infty,1},\ldots,F_{\infty,4}$ such that
$F_{\infty,1},\ldots,F_{\infty,4}$ intersect $C$ transversally while
$F_{0,1}$ and $F_{0,2}$ intersects $L$ but not $C$, we obtain
another contraction morphism $\tilde{\mu}:V\rightarrow\mathbb{P}^{2}$
for which the proper transforms of $C$ and $L$ are respectively
a conic and its tangent line at the point $\tilde{\mu}(p)$. One checks
that the lift to $V$ of the rational pencil on $\mathbb{P}^{2}$
generated by $\tilde{\mu}_{*}(C)$ and $2\mu_{*}(L)$ restricts on
$S_{R}$ to an $\mathbb{A}^{1}$-fibration $\pi'_{R}:S_{R}\rightarrow\mathbb{P}^{1}$
with two reducible degenerate fibers: one consisting of the disjoint
union of the curves $F_{0,i}\cap S\simeq\mathbb{A}^{1}$, $i=1,2$,
both occuring with multiplicity two and the other one consisting of
the disjoint union of the reduced curves $F_{\infty,i}\cap S_{R}\simeq\mathbb{A}^{1}$,
$i=1,\ldots,4$. The description of the degenerate fibers shows that
this second $\mathbb{A}^{1}$-fibration is not isomorphic to the one
$\pi_{R}:S_{R}\rightarrow\mathbb{P}^{1}$, so that $S_{R}$ carries
at least two distinct types of $\mathbb{A}^{1}$-fibrations over $\mathbb{P}^{1}$. 
\end{rem}
\begin{parn} To determine the automorphism group of $S_{R}=V\setminus D$
we first notice that the subgroup $\mathrm{Aut}(V,D)$ of $\mathrm{Aut}(V)$
consisting of automorphisms of $V$ which leave $D$ globally invariant
can be identified in a natural way with a subgroup of $\mathrm{Aut}(S_{R})$.
The latter is always finite, and even trivial if the cubic surface
$V$ is chosen general. On the other hand, $S_{R}$ admits at least
another natural automorphism which is obtained as follows: the projection
$\mathbb{P}^{3}\dashrightarrow\mathbb{P}^{2}$ from the point $p=L\cap C$
induces a rational map $Q\dashrightarrow\mathbb{P}^{2}$ with $p$
as a unique proper base point and whose lift to the blow-up $\alpha:W\rightarrow V$
of $V$ at $p$ coincides with the morphism $\theta:W\rightarrow\mathbb{P}^{2}$
defined by the anticanonical linear system $|-K_{W}|$. The latter
factors into a birational morphism $W\rightarrow Y$ contracting the
proper transform of $L$ followed by a Galois double cover $Y\rightarrow\mathbb{P}^{2}$
ramified over an irreducible quartic curve $\Delta$ with a unique
double point located at the image of $L$. The non trivial involution
of the double cover $Y\rightarrow\mathbb{P}^{2}$ induces an involution
$G_{W}:W\rightarrow W$ fixing $L$ and exchanging the proper transform
of $C$ with the exceptional divisor $E$ of $\alpha$. The former
descends to a birational involution $G_{V,p}:V\dashrightarrow V$
which restricts further to a biregular involution $j_{G_{V,p}}$ of
$S_{R}=V\setminus D$.

\end{parn}

\noindent The following description of the automorphism group of
$S_{R}$ shows in particular that these surfaces are rigid: 
\begin{lem}
For a surface $S_{R}=V\setminus D$ as above, there exists a split
exact sequence 
\[
0\rightarrow\mathrm{Aut}(V,D)\rightarrow\mathrm{Aut}(S_{R})\rightarrow\mathbb{Z}_{2}\cdot j_{G_{V,p}}\rightarrow0.
\]
\end{lem}
\begin{proof}
We interpret every automorphism of $S_{R}$ as a birational self-map
$f:V\dashrightarrow V$ of $V$ restricting to an isomorphism from
$S_{R}=V\setminus D$ to itself. Since $f\in\mathrm{Aut}(V,D)$ in
case it is biregular, it is enough to show that either $f$ or $G_{V,p}\circ f$
is biregular. To establish this alternative, it suffices to check
that the lift $f_{W}=\alpha^{-1}f\alpha:W\dashrightarrow W$ of $f$
to $W$ is a biregular morphism, hence an automorphism of the pair
$(W,\alpha^{-1}(D)_{\mathrm{red}})$. Indeed if so, then $f_{W}$
preserves the union of $E$ and the proper transform of $D$ as these
are the only $\left(-1\right)$-curves contained in the support of
$\alpha^{-1}(D)_{\mathrm{red}}$. Since by construction $G_{W}$ exchanges
$E$ and the proper transform of $D$, it follows that either $f_{W}$
or $G_{W}\circ f_{W}$ leaves $E$, the proper transform of $D$ and
the proper transform of $L$ invariant. This implies in turn that
either $f=\alpha f_{W}\alpha^{-1}$ or $\alpha G_{W}\circ f_{W}\alpha^{-1}=(\alpha G_{W}\alpha^{-1})\circ(\alpha f_{W}\alpha^{-1})=G_{V,p}\circ f$
is a biregular automorphism of $V$. 

To show that $f_{W}$ is a biregular automorphism of $W$, we consider
the lift $\tilde{f}=\sigma^{-1}\circ f\circ\sigma:\tilde{V}\dashrightarrow\tilde{V}$
of $f$ to the variety $\tilde{\alpha}:\tilde{V}\rightarrow W$ obtained
from $W$ by blowing-up further the intersection point of $E$ and
of the proper transform of $C$, say with exceptional divisor $\tilde{E}$.
We identify $S_{R}$ with the complement in $\tilde{V}$ of the SNC
divisor $\tilde{D}=L\cup C\cup E\cup\tilde{E}$. Now suppose by contradiction
that $\tilde{f}$ is strictly birational and consider its minimal
resolution $\tilde{V}\stackrel{\beta}{\leftarrow}X\stackrel{\beta'}{\rightarrow}\tilde{V}'$.
Recall that the minimality of the resolution implies in particular
that there is no $\left(-1\right)$-curve in $X$ which is exceptional
for $\beta$ and $\beta'$ simultaneously. Furthermore, since $\tilde{V}$
is smooth and $\tilde{D}$ is an SNC divisor, $\beta'$ decomposes
into a finite sequence of blow-downs of successive $\left(-1\right)$-curves
supported on the boundary $B=\beta^{-1}(\tilde{D})_{\mathrm{red}}=(\beta')^{-1}(\tilde{D})_{\mathrm{red}}$
with the property that at each step, the proper transform of $B$
is again an SNC divisor. The structure of $\tilde{D}$ implies that
the only possible $\left(-1\right)$-curve in $B$ which is not exceptional
for $\beta$ is the proper transform of $\tilde{E}$, but after its
contraction, the proper transform of $B$ would no longer be an SNC
divisor, a contradiction. So $\tilde{f}:\tilde{V}\rightarrow\tilde{V}$
is a morphism and the same argument shows that it does not contract
any curve in the boundary $\tilde{D}$. Thus $\tilde{f}$ is a biregular
automorphism of $\tilde{V}$, in fact, an element of $\mathrm{Aut}(\tilde{V},\tilde{D})$.
Since $\tilde{E}$ is the unique $\left(-1\right)$-curve contained
in the support of $\tilde{D}$ it must be invariant by $\tilde{f}$
which implies in turn that $f_{W}=\tilde{\alpha}\tilde{f}\tilde{\alpha}^{-1}$
is a biregular automorphism of the pair $(W,\alpha^{-1}(D)_{\mathrm{red}})$,
as desired. 
\end{proof}

\subsection{Flexible mates }

In this subsection, we construct $1$-flexible affine surfaces $S_{F}$
admitting $\mathbb{A}^{1}$-fibrations $\pi_{F}:S_{F}\rightarrow\mathbb{P}^{1}$
whose degenerate fibers resemble the ones of the fibrations $\pi_{R}:S_{R}\rightarrow\mathbb{P}^{1}$
described in \S \ref{par:Sr-Fib} above. A more precise interpretation
of this resemblance, going beyond the bare fact that the number of
their irreducible components and their respective multiplicities are
the same, will be given in the next section. 

\begin{parn} \label{par:Sf_Fib} For the construction, we start with
a Hirzebruch surface $\pi_{n}:\mathbb{F}_{n}=\mathbb{P}(\mathcal{O}_{\mathbb{P}^{1}}\oplus\mathcal{O}_{\mathbb{P}^{1}}(-n))\rightarrow\mathbb{P}^{1}$,
$n\geq0$, in which we fix  an ample section $C\simeq\mathbb{P}^{1}$
of $\pi_{n}$ and two distinct fibers, say $F_{0}=\pi_{n}^{-1}(p_{0})$
and $F_{\infty}=\pi_{n}^{-1}(p_{\infty})$, where $p_{0},p_{\infty}\in\mathbb{P}^{1}$
. We let $\sigma:X\rightarrow\mathbb{F}_{n}$ be the birational map
obtained by the following sequence of blow-ups:

- Step 1 consists of the blow-up of five distinct points $p_{\infty,1},\ldots,p_{\infty,5}$
on $F_{\infty}\setminus C$ with respective exceptional divisors $G_{\infty,1},\ldots,G_{\infty,5}$. 

- Step 2 consists of the blow-up of a point $p_{0,1}\in F_{0}\setminus C$
with exceptional divisor $E_{1}$, followed by the blow-up of the
intersection point $p_{0,2}$ of the proper transform of $F_{0}$
with $E_{1}$, with exceptional divisor $E_{2}$, then followed by
the blow-up of the intersection point $p_{0,3}$ of the proper transform
of $F_{0}$ with $E_{2}$, with exceptional divisor $E_{3}$. Finally,
we blow-up a point $p_{0,4}\in E_{3}$ distinct from the intersection
points of $E_{3}$ with the proper transforms of $F_{0}$ and $E_{2}$
respectively. We denote the last exceptional divisor produced by $G_{0,1}$. 

The structure morphism $\pi_{n}:\mathbb{F}_{n}\rightarrow\mathbb{P}^{1}$
lifts to a $\mathbb{P}^{1}$-fibration $\overline{p}=\pi_{n}\circ\sigma:X\rightarrow\mathbb{P}^{1}$
with two degenerate fibers $\overline{p}^{-1}(p_{0})=F_{0}+E_{1}+2E_{2}+3E_{3}+3G_{0,1}$
and $\overline{p}^{-1}(p_{\infty})=F_{\infty}+\sum_{i=1}^{5}G_{\infty,i}$.
The inverse image by $\sigma$ of the divisor $F_{0}\cup C\cup F_{\infty}$
is pictured in Figure \ref{fig:FibSF}. Letting $S_{F}$ be the open
complement in $X$ of the divisor $B=F_{\infty}\cup C\cup F_{0}\cup E_{3}\cup E_{2}\cup E_{1}$,
the restriction of $\overline{p}$ to $S_{F}$ is an $\mathbb{A}^{1}$-fibration
$\pi_{F}:S_{F}\rightarrow\mathbb{P}^{1}$ with two degenerate fibers:
the one $\pi_{F}^{-1}(p_{0})$ is irreducible of multiplicity three
consisting of the intersection of $G_{0,1}$ with $S_{F}$ and the
other one $\pi_{F}^{-1}(p_{\infty})$ is reduced, consisting of the
disjoint union of the curves $G_{\infty,i}\cap S_{F}\simeq\mathbb{A}^{1}$,
$i=1,\ldots,5$. 

\begin{figure}[ht]
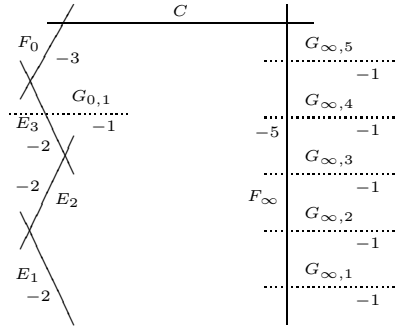
 
$$\xygraph{!{<0cm,0cm>;<0.7cm,0cm>:<0cm,0.5cm>::}
!{(2,-2)}="C1" !{(2,6.5)}="C2"  
!{(1.6,-1)}="F11" !{(1.6,0.5)}="F21" !{(1.6,2)}="F31" !{(1.6,3.5)}="F41" !{(1.6,5)}="F51"
!{(4,-1)}="F12" !{(4,0.5)}="F22" !{(4,2)}="F32" !{(4,3.5)}="F42" !{(4,5)}="F52"
!{(2.5,6)}="S1"!{(-2.5,6)}="S2"
!{(-2,6.5)}="F01" !{(-3,4)}="F02"
!{(-3,5)}="E31" !{(-2,2)}="E32"
!{(-2,3)}="E21" !{(-3,0)}="E22"
!{(-3,1)}="E11" !{(-2,-2)}="E12"
!{(-3.2,3.6)}="G1" !{(-1,3.6)}="G2"
"F11"-@{.}^{G_{\infty,1}}_>(0.8){-1}"F12"
"F21"-@{.}^{G_{\infty,2}}_>(0.8){-1}"F22"
"F31"-@{.}^{G_{\infty,3}}_>(0.8){-1}"F32"
"F41"-@{.}^{G_{\infty,4}}_>(0.8){-1}"F42"
"F51"-@{.}^{G_{\infty,5}}_>(0.8){-1}"F52"
"C1"-^>(0.4){F_{\infty}}^>(0.6){-5}"C2"
"S1"-_{C}"S2"
"F01"-_{F_0}^{-3}"F02"
"E31"-_{E_3}_>(0.7){-2}"E32"
"E21"-^{E_2}_{-2}"E22"
"E11"-_{E_1}_>(0.7){-2}"E12"
"G1"-@{.}^>(0.7){G_{0,1}}_>(0.8){-1}"G2"
}
$$
\caption{The total transform of $F_0 \cup C \cup F_{\infty} \subset \mathbb{F}_n$ in $X$. The plain curves correspond to irreducible components of the boundary divisor $B$.} 
\label{fig:FibSF} 
\end{figure} 

\end{parn}
\begin{lem}
\label{lem:Flex_Model} A surface $S_{F}=X\setminus B$ as above is
affine and $1$-flexible.\end{lem}
\begin{proof}
By construction, $B$ is chain of smooth complete rational curves.
So the $1$-flexibility of $S_{F}$ will follows from Theorem \ref{thm:Giz-Flex}
provided that $S_{F}$ is indeed affine and has no non constant invertible
functions. Since $\pi_{F}:S_{F}\rightarrow\mathbb{P}^{1}$ is an $\mathbb{A}^{1}$-fibration,
an invertible function on $S_{F}$ is constant in restriction to every
non degenerate fiber of $\pi_{F}$ and hence has the form $f\circ\pi_{F}$
for a certain global invertible function $f$ on $\mathbb{P}^{1}$.
So such a function is certainly constant. To establish the affineness
of $S_{F}$, we first observe that $S_{F}$ does not contain a complete
curve. Indeed, otherwise since the points blown-up by $\sigma:X\rightarrow\mathbb{F}_{n}$
are contained in $\mathbb{F}_{n}\setminus C$, the image by $\sigma$
of such curve would be a complete curve in $\mathbb{F}_{n}$ which
does not intersect $C$, in contradiction with the ampleness of $C$
in $\mathbb{F}_{n}$. On the other hand, since $C$ has positive self-intersection
in $\mathbb{F}_{n}$, whence in $X$, one checks by direct computation
that for $a_{\infty},a,a_{0},a_{1},a_{2},a_{3}\in\mathbb{Z}_{>0}$
such that $a_{0}\gg a_{3}\gg a_{2}\gg a_{1}$ and $a\gg\max(a_{0},a_{\infty})$,
the effective divisor $\tilde{B}=a_{\infty}F_{\infty}+aC+a_{0}F_{0}+a_{1}E_{1}+a_{2}E_{2}+a_{3}E_{3}$
has positive self-intersection and positive intersection which each
of its irreducible components. It then follows from the Nakai-Moishezon
criterion that $\tilde{B}$ is an ample effective divisor supported
on $B$, and hence that $S_{F}=X\setminus B$ is affine. \end{proof}
\begin{rem}
\label{rk:Alt-SfFib} In the construction of \S \ref{par:Sf_Fib},
one can replace Step 1 and 2 by the following alternative sequence
of blow-ups $\sigma':X'\rightarrow\mathbb{F}_{n}$:

- Step 1' consists of the blow-up of four distinct points $p'_{\infty,1},\ldots,p'_{\infty,4}$
on $F_{\infty}\setminus C$ with respective exceptional divisors $G'_{\infty,1},\ldots,G'_{\infty,4}$. 

- Step 2' consists of the blow-up of a point $p'_{0,1}\in F_{0}\setminus C$
with exceptional divisor $E_{1}'$, followed by the blow-up of the
intersection point $p'_{0,2}$ of the proper transform of $F_{0}$
with $E'_{1}$, with exceptional divisor $E'_{2}$, then followed
by the blow-up of a pair of distinct points $p'_{0,3}$ and $p''_{0,3}$
on $E_{2}'$ distinct from the intersection points of $E_{2}'$ with
the proper transforms of $F_{0}$ and $E_{1}'$, with respective exceptional
divisors $G'_{0,1}$ and $G'_{0,2}$. 

The morphism $\overline{p}'=\pi_{n}\circ\sigma':X'\rightarrow\mathbb{P}^{1}$
is then a $\mathbb{P}^{1}$-fibration with two degenerate fibers ${\overline{p}'}^{-1}(p_{0})=F_{0}+E'_{1}+2E'_{2}+2G'_{0,1}+2G'_{0,2}$
and ${\overline{p}'}^{-1}(p_{\infty})=F_{\infty}+\sum_{i=1}^{4}G'_{\infty,i}$.
The same argument as in the proof of Lemma \ref{lem:Flex_Model} above
shows that the complement in $X'$ of the chain of smooth complete
rational curves $B'=F_{\infty}\cup C\cup F_{0}\cup E_{1}'\cup E_{2}'$
is a $1$-flexible affine surface, on which $\overline{p}'$ restricts
to an $\mathbb{A}^{1}$-fibration $\pi'_{F}:S'_{F}\rightarrow\mathbb{P}^{1}$
with two degenerate fibers consisting respectively of the disjoint
union of $G'_{0,i}\cap S'_{F}\simeq\mathbb{A}^{1}$, $i=1,2$ both
occurring with multiplicity $2$ and of the disjoint union of the
reduced curves $G'_{\infty,i}\cap S'_{F}\simeq\mathbb{A}^{1}$, $i=1,\ldots,4$.
So $\pi'_{F}:S'_{F}\rightarrow\mathbb{P}^{1}$ resembles the alternative
$\mathbb{A}^{1}$-fibration $\pi'_{R}:S_{R}\rightarrow\mathbb{P}^{1}$
described in Remark \ref{Rk:Alt-SrFib} above. 
\end{rem}

\section{rigidity lost}

\noindent The last ingredient needed to derive Theorem \ref{thm:Main}
is the following result:
\begin{prop}
\label{prop:Factor} Let $\pi_{R}:S_{R}\rightarrow\mathbb{P}^{1}$
and $\pi_{F}:S_{F}\rightarrow\mathbb{P}^{1}$ be a pair of $\mathbb{A}^{1}$-fibered
surfaces as constructed in \S \ref{par:Sr-Fib} and \ref{par:Sf_Fib}
above. Then there exists an algebraic space $\delta:\mathfrak{C}\rightarrow\mathbb{P}^{1}$
such that $\pi_{R}$ and $\pi_{F}$ factor respectively through \'etale
locally trivial $\mathbb{A}^{1}$-bundles $\rho_{R}:S_{R}\rightarrow\mathfrak{C}$
and $\rho_{F}:S_{F}\rightarrow\mathfrak{C}$. 
\end{prop}
\noindent Let us first explain how derive the $1$-flexibility of
the cylinder $S_{F}\times\mathbb{A}^{2}$ from this Proposition. 

\begin{parn} Recall that since the automorphism group of $\mathbb{A}^{1}$
is the affine group $\mathrm{Aff}_{1}=\mathbb{G}_{m}\ltimes\mathbb{G}_{a}$,
every \'etale locally trivial $\mathbb{A}^{1}$-bundle $\rho:S\rightarrow\mathfrak{C}$
is in fact an affine-linear bundle. This means that there exists a
line bundle $\mathrm{p}:L\rightarrow\mathfrak{C}$ such that $\rho:S\rightarrow\mathfrak{C}$
has the structure of an \'etale $L$-torsor, that is, an \'etale
locally trivial principal homogeneous bundle under $L$, considered
as a group space over $\mathfrak{C}$ for the group law induced by
the addition of germs of sections. Isomorphy classes of such principal
homogeneous $L$-bundles are then classified by the cohomology group
$H_{\mathrm{\acute{e}t}}^{1}(\mathfrak{C},L)$ (see e.g. \cite[\S 1.2]{Dub11}). 

\end{parn}

\begin{parn} So Proposition \ref{prop:Factor} implies in particular
that $\rho_{R}:S_{R}\rightarrow\mathfrak{C}$ and $\rho_{F}:S_{F}\rightarrow\mathfrak{C}$
can be equipped with the structure of principal homogeneous bundles
under suitable line bundles $\mathrm{p}_{R}:L_{R}\rightarrow\mathfrak{C}$
and $\mathrm{p}_{F}:L_{F}\rightarrow\mathfrak{C}$ respectively. As
a consequence, the fiber product $Z=S_{R}\times_{\mathfrak{C}}S_{F}$
is simultaneously equipped via the first and second projection with
the structure of a principal homogeneous bundle under the line bundles
$\rho_{R}^{*}L_{F}$ and $\rho_{F}^{*}L_{R}$ respectively. But since
$S_{R}$ and $S_{F}$ are both affine, the vanishing of $H_{\mathrm{\acute{e}t}}^{1}(S_{R},\rho_{R}^{*}L_{F})$
and $H_{\mathrm{\acute{e}t}}^{1}(S_{F},\rho_{F}^{*}L_{R})$ implies
that $\mathrm{pr}_{1}:Z\rightarrow S_{R}$ and $\mathrm{pr}_{2}:Z\rightarrow S_{F}$
are the trivial $\rho_{R}^{*}L_{F}$-torsor and $\rho_{F}^{*}L_{R}$-torsor
respectively. In other word, $Z$ carries simultaneously the structure
of a line bundle over $S_{R}$ and $S_{F}$. 

\end{parn}

\begin{parn} \label{par:end_of_proof} Now since $S_{F}$ is $1$-flexible
by virtue of Theorem \ref{thm:Giz-Flex}, we deduce from Lemma \ref{lem:Flex-lift}
that $Z$ is $1$-flexible. Furthermore, the same Lemma implies that
given any line bundle $\mathrm{p}:Z'\rightarrow S_{R}$, the total
space of the rank $2$ vector bundle $\mathrm{pr}_{1}\times\mathrm{p}:E=Z'\times_{X}Z\rightarrow S_{R}$
over $S_{R}$ is $1$-flexible. On the other hand, it follows from
\cite[Theorem 3.2]{Mur69} that every rank $2$ vector bundle $E\rightarrow S_{R}$
splits a trivial factor, whence is isomorphic to the direct sum of
its determinant $\det E$ and of the trivial line bundle. Choosing
for $Z'$ a line bundle representing the inverse of the class of $\mathrm{pr}_{1}:Z\rightarrow S_{R}$
in the Picard group $\mathrm{Pic}(S_{R})$ of $S_{R}$ yields a vector
bundle $E=Z'\times_{X}Z\rightarrow S_{R}$ with trivial determinant,
whence isomorphic to the trivial one $S_{R}\times\mathbb{A}^{2}$,
and with $1$-flexible total space. 

\end{parn}

\subsection{Proof of Proposition \ref{prop:Factor} }

\begin{parn} To prove Proposition \ref{prop:Factor}, we first observe
that if it exists, an algebraic space $\delta:\mathfrak{C}\rightarrow\mathbb{P}^{1}$
with the property that a given $\mathbb{A}^{1}$-fibration $\pi:S\rightarrow\mathbb{P}^{1}$
on a smooth surface $S$ factors as $\delta\circ\rho$, where $\rho:S\rightarrow\mathfrak{C}$
is an \'etale locally trivial $\mathbb{A}^{1}$-bundle, is unique
up to isomorphism of spaces over $\mathbb{P}^{1}$. Indeed, suppose
that $\delta':\mathfrak{C}\rightarrow\mathbb{P}^{1}$ is another such
space for which we have $\pi=\delta'\circ\rho'$ where $\rho':S'\rightarrow\mathfrak{C}'$
is an \'etale locally trivial $\mathbb{A}^{1}$-bundle. The closed
fibers of $\rho$ and $\rho'$ being both in one-to-one correspondence
with irreducible components of closed fibers of $\pi$, it follows
that for every closed point $c\in\mathfrak{C}$ there exists a unique
closed point $c'\in\mathfrak{C}'$ such that $\delta(c)=\delta'(c')$
and $\rho^{-1}(c)=(\rho')^{-1}(c')\subset\pi^{-1}(\delta(c))$. So
the correspondence $c\mapsto c'$ defines a bijection $\psi:\mathfrak{C}\rightarrow\mathfrak{C}'$
such that $\rho'=\psi\circ\rho$ and $\delta=\delta'\circ\psi$. Letting
$f:C\rightarrow\mathfrak{C}$ be an \'etale cover over which $\rho:S\rightarrow\mathfrak{C}$
becomes trivial, say with isomorphism $\theta:S\times_{\mathfrak{C}}C\stackrel{\sim}{\rightarrow}C\times\mathbb{A}^{1}$,
and choosing a section $\sigma:C\rightarrow C\times\mathbb{A}^{1}$
of $\mathrm{pr}_{C}:C\times\mathbb{A}^{1}\rightarrow C$, the composition
$\psi\circ f=\psi\circ f\circ\mathrm{pr}_{C}\circ\sigma$ is equal
to $\psi\circ\rho\circ\mathrm{pr}_{1}\circ\theta^{-1}\circ\sigma$
and hence to $\rho'\circ\mathrm{pr}_{1}\circ\theta^{-1}\circ\sigma$
by construction of $\psi$. \[\xymatrix{ C\times \mathbb{A}^1 \ar[r]^{\theta^{-1}} \ar[dr]^{\mathrm{pr}_C} & S\times_{\mathfrak{C}}C \ar[r]^{\mathrm{pr}_1} \ar[d]^-{\mathrm{pr}_2} & S \ar[d]^{\rho} \ar[dr]^{\rho'}& \\ & C \ar@<1ex>[ul]^{\sigma} \ar[r]^f & \mathfrak{C} \ar[r]^{\psi} & \mathfrak{C}' }\]This
implies that $\psi\circ f:C\rightarrow\mathfrak{C}'$ is a morphism
whence that $\psi:\mathfrak{C}\rightarrow\mathfrak{C}'$ is a morphism
since being a morphism is a local property with respect to the \'etale
topology. The same argument on an \'etale cover $f':C'\rightarrow\mathfrak{C}'$
over which $\rho':S'\rightarrow\mathfrak{C}'$ becomes trivial implies
that the set-theoretic inverse $\psi^{-1}$ of $\psi$ is also a morphism,
and so $\psi:\mathfrak{C}\rightarrow\mathfrak{C}'$ is an isomorphism
of spaces over $\mathbb{P}^{1}$. 

\end{parn}

\begin{parn} In what follows, given an $\mathbb{A}^{1}$-fibered
surface $\pi:S\rightarrow\mathbb{P}^{1}$, we use the notation $S/\mathbb{A}^{1}$
to refer to an algebraic space $\delta:\mathfrak{C}\rightarrow\mathbb{P}^{1}$
with the property that $\pi$ factors through an \'etale locally
trivial $\mathbb{A}^{1}$-bundle $\rho:S\rightarrow\mathfrak{C}$.
The previous observation implies that its existence is a local problem
with respect to the Zariski topology on $\mathbb{P}^{1}$. More precisely,
we may cover $\mathbb{P}^{1}$ by finitely many affine open subsets
$U_{i}$, $i=1,\ldots,r$ over which the restriction of $\pi:S\rightarrow\mathbb{P}^{1}$
is an $\mathbb{A}^{1}$-fibration with a most a degenerate fiber,
say $\pi^{-1}(p_{i})$ for some $p_{i}\in U_{i}$. Since the restriction
of $\pi$ over $U_{i,*}=U_{i}\setminus\left\{ p_{i}\right\} $ is
then a Zariski locally trivial $\mathbb{A}^{1}$-bundle, we see that
if $\delta_{i}:\mathfrak{C}_{i}=\pi^{-1}(U_{i})/\mathbb{A}^{1}\rightarrow U_{i}$
exists then the restriction of $\delta_{i}$ over $U_{i}\setminus\left\{ p_{i}\right\} $
is an isomorphism of schemes over $U_{i}$. This implies that the
isomorphisms $\delta_{j}^{-1}\circ\delta_{i}\mid_{\delta^{-1}(U_{i,*}\cap U_{j,*})}:\delta_{i}^{-1}(U_{i,*}\cap U_{j,*})\rightarrow\delta_{j}^{-1}(U_{i,*}\cap U_{j,*})$,
$i,j=1,\ldots,r$, satisfy the usual cocyle condition on triple intersections
whence that the algebraic space $\delta:\mathfrak{C}=S/\mathbb{A}^{1}\rightarrow\mathbb{P}^{1}$
with the desired property is obtained by gluing the local ones $\delta_{i}:\mathfrak{C}_{i}\rightarrow U_{i}$,
$i=1,\ldots,r$ along their respective open sub-schemes $\delta_{i}^{-1}(U_{i,*}\cap U_{j,*})\subset\mathfrak{C}_{i}$,
$i,j=1,\ldots,r$ via these isomorphisms. 

\end{parn}

\begin{parn} Now we turn more specifically to the case of the $\mathbb{A}^{1}$-fibrations
$\pi_{R}:S_{R}\rightarrow\mathbb{P}^{1}$ and $\pi_{F}:S_{F}\rightarrow\mathbb{P}^{1}$
constructed in \S \ref{par:Sr-Fib} and \S \ref{par:Sf_Fib} respectively.
Both have exactly two degenerate fibers, one irreducible of multiplicity
three and the other one consisting of the disjoint union of five reduced
curves. So up to an automorphism of $\mathbb{P}^{1}$ we may choose
a pair of distinct point $p_{0},p_{\infty}\in\mathbb{P}^{1}$ such
that $\pi_{R}^{-1}(p_{0})=3T\cap S_{R}$, $\pi_{F}^{-1}(p_{0})=3G_{0,1}\cap S_{F}$,
$\pi_{R}^{-1}(p_{\infty})=\bigsqcup_{i=1}^{5}F_{i}\cap S_{R}$ and
$\pi_{F}^{-1}(p_{\infty})=\bigsqcup_{i=1}^{5}G_{\infty,i}\cap S_{F}$.
Letting $U_{0}=\mathbb{P}^{1}\setminus\{p_{\infty}\}$ and $U_{\infty}=\mathbb{P}^{1}\setminus\{p_{0}\}$,
the existence and isomorphy of the algebraic spaces $\pi_{R}^{-1}(U_{0})/\mathbb{A}^{1}$
and $\pi_{F}^{-1}(U_{0})/\mathbb{A}^{1}$ (resp. $\pi_{R}^{-1}(U_{\infty})/\mathbb{A}^{1}$
and $\pi_{F}^{-1}(U_{\infty})/\mathbb{A}^{1}$) hence of those $S_{R}/\mathbb{A}^{1}$
and $S_{F}/\mathbb{A}^{1}$, follows from a reinterpretation of a
description due to Fieseler \cite{Fie94}: 

- Since the unique degenerate fiber of the restriction of $\pi_{R}$
(resp. $\pi_{F}$) over $U_{\infty}$ is reduced, consisting of five
irreducible components, $\pi_{R}^{-1}(U_{\infty})/\mathbb{A}^{1}$
and $\pi_{F}^{-1}(U_{\infty})/\mathbb{A}^{1}$ are isomorphic to the
scheme $\delta_{\infty}:\mathfrak{C}_{\infty}\rightarrow U_{\infty}$
obtained from $U_{\infty}$ by replacing the point $p_{\infty}$ by
five copies $p_{\infty,1},\ldots,p_{\infty,5}$ of itself, one for
each irreducible component of $\pi_{R}^{-1}(p_{\infty})$ (resp. $\pi_{F}^{-1}(p_{\infty})$).
More explicitly, $\pi_{R}$ restricts on $S_{R,\infty,i}=\pi_{R}^{-1}(U_{\infty})\setminus\bigsqcup_{j\neq i}(F_{j}\cap S_{R})$,
$i=1,\ldots,5$, to an $\mathbb{A}^{1}$-fibration $\pi_{R,\infty,i}:S_{R,\infty,i}\rightarrow U_{\infty}$
with no degenerate fiber over the factorial base $U_{\infty}\simeq\mathbb{A}^{1}$
and hence is a trivial $\mathbb{A}^{1}$-bundle. So $S_{R,\infty,i}/\mathbb{A}^{1}\simeq U_{\infty}$
and $\pi_{R}^{-1}(U_{\infty})/\mathbb{A}^{1}$ is thus isomorphic
to the $U_{\infty}$-scheme obtained by gluing $r$-copies $\delta_{\infty,i}:U_{\infty,i}\stackrel{\sim}{\rightarrow}U_{\infty}$,
$i=1,\ldots,5$, of $U_{\infty}$ by the identity along the open subsets
$U_{\infty,i}\setminus\{p_{\infty,i}\}$, where $p_{\infty,i}=\delta_{\infty,i}^{-1}(p_{\infty})$.
The same description holds for $\pi_{F}^{-1}(U_{\infty})/\mathbb{A}^{1}$. 

- The situation for the open subsets $S_{R,0}=\pi_{R}^{-1}(U_{0})$
and $S_{F,0}=\pi_{F}^{-1}(U_{0})$ is a little more complicated. Letting
$g:\tilde{U}_{0}\rightarrow U_{0}$ be a Galois cover of order three
ramified over $p_{0}$ and \'etale everywhere else, the inverse image
of $\pi_{F}^{-1}(p_{0})_{\mathrm{red}}$ in the normalization $\tilde{S}_{R,0}$
of the reduced fiber product $(S_{R}\times_{U_{0}}\tilde{U}_{0})_{\mathrm{red}}$
is the disjoint union of three curves $\ell_{0,1}$, $\ell_{0,\varepsilon}$,
and $\ell_{0,\varepsilon^{2}}$ (where $\varepsilon\in\mathbb{C}^{*}$
is a primitive cubic root of unity) which are permuted by the action
of the Galois group $\mu_{3}$ of cubic roots of unity. The $\mathbb{A}^{1}$-fibration
$\mathrm{pr}_{1}:S_{R,0}\times_{U_{0}}\tilde{U}_{0}\rightarrow U'_{0}$
lifts to one $\tilde{\pi}{}_{R,0}:\tilde{S}_{R,0}\rightarrow\tilde{U}_{0}$
with a unique, reduced, degenerate fiber $(\pi')_{R,0}^{-1}(\tilde{p}_{0})$,
where $\tilde{p}_{0}=g^{-1}(p_{0})$, which consists of the union
of the $\ell_{0,\alpha}$, $\alpha=1,\varepsilon,\varepsilon^{2}$.
The same argument as in the previous case implies then that $\tilde{\mathfrak{C}}_{0}=\tilde{S}_{R,0}/\mathbb{A}^{1}$
is isomorphic to the $\tilde{U}_{0}$-scheme $\tilde{\delta}_{0}:\tilde{\mathfrak{C}}_{0}\rightarrow\tilde{U}_{0}$
obtained by gluing three copies $\tilde{\delta}{}_{0,\alpha}:\tilde{U}{}_{0,\alpha}\stackrel{\sim}{\rightarrow}\tilde{U}_{0}$,
$\alpha=1,\varepsilon,\varepsilon^{2}$, of $\tilde{U}_{0}$ by the
identity outside the points $\tilde{p}{}_{0,\alpha}=(\tilde{\delta}{}_{0,\alpha})^{-1}(\tilde{p}{}_{0})$.
Furthermore, the action of the Galois group $\mu_{3}$ on $\tilde{S}{}_{R,0}$
descends to a fixed point free action on $\tilde{\mathfrak{C}}_{0}$
defined locally by $\tilde{U}{}_{0,\alpha}\ni\tilde{p}\mapsto\varepsilon\cdot\tilde{p}\in\tilde{U}{}_{0,\varepsilon\alpha}$.
A geometric quotient for this action on $\tilde{\mathfrak{C}}{}_{0}$
exists in the category of algebraic spaces in the form of an \'etale
$\mu_{3}$-torsor $\tilde{\mathfrak{C}}_{0}\rightarrow\tilde{\mathfrak{C}}_{0}/\mu_{3}$
over a certain algebraic space $\tilde{\mathfrak{C}}_{0}/\mu_{3}$
and we obtain a commutative diagram \[\xymatrix{\tilde{S}_{R,0} \ar[r] \ar[d]_{\tilde{\rho}_{R,0}} & \tilde{S}_{R,0}/\mu_3\simeq S_{R,0} \ar[d]^{\rho_{R,0}} \\ \tilde{\mathfrak{C}}_{0} \ar[r] \ar[d]_{\tilde{\delta}_0} & \tilde{\mathfrak{C}}_{0}/\mu_3 \ar[d]^{\delta_0} \\ \tilde{U}_0 \ar[r] & \tilde{U}_0/\mu_3\simeq U_0 }\]in
which the top square is cartesian. It follows that the induced morphism
$\rho_{R,0}:S_{R,0}\rightarrow\tilde{\mathfrak{C}}_{0}/\mu_{3}$ is
an \'etale locally trivial $\mathbb{A}^{1}$-bundle which factors
the restriction of $\pi_{R}$ to $S_{R,0}$. So $\delta_{0}:\tilde{\mathfrak{C}}_{0}/\mu_{3}\rightarrow U_{0}$
is the desired algebraic space $S_{R,0}/\mathbb{A}^{1}$.\\

It is clear from the construction that the isomorphy type of $\tilde{\mathfrak{C}}_{0}/\mu_{3}$
as a space over $U_{0}$ depends only on the fact that $S_{R,0}$
is smooth and that $\pi_{R}\mid_{S_{R,0}}:S_{R,0}\rightarrow U_{0}$
is an $\mathbb{A}^{1}$-fibration with a unique degenerate fiber of
multiplicity three over $p_{0}$, and not on the full isomorphy type
of $S_{R,0}$ as a scheme over $U_{0}$. In other word, the same construction
applied $\pi_{F}\mid_{S_{F,0}}:S_{F,0}\rightarrow U_{0}$ yields an
algebraic space $S_{F,0}/\mathbb{A}^{1}$ which is isomorphic to $\tilde{\mathfrak{C}}_{0}/\mu_{3}$
as spaces over $U_{0}$. \\

Finally, the desired algebraic space $\mathfrak{C}=S_{R}/\mathbb{A}_{1}=S_{F}/\mathbb{A}_{1}$
is obtained by gluing $\mathfrak{C}_{\infty}$ and $\mathfrak{C}_{0}=\tilde{\mathfrak{C}}_{0}/\mu_{3}$
by the identity along the open sub-schemes $\delta_{\infty}^{-1}(U_{0}\cap U_{\infty})\simeq U_{0}\cap U_{\infty}\simeq\delta_{0}^{-1}(U_{0}\cap U_{\infty})$.
This completes the proof of Proposition \ref{prop:Factor}. 

\end{parn}
\begin{rem}
A similar construction applies to the $\mathbb{A}^{1}$-fibrations
$\pi'_{R}:S_{R}\rightarrow\mathbb{P}^{1}$ and $\pi'_{F}:S_{F}\rightarrow\mathbb{P}^{1}$
considered in remarks \ref{Rk:Alt-SrFib} and \ref{rk:Alt-SfFib}
respectively. The desired algebraic space $\mathfrak{C}'=S_{R}/\mathbb{A}^{1}=S_{F}/\mathbb{A}^{1}$
is again obtained as the gluing by the identity along ${\delta'}_{\infty}^{-1}(U_{0}\cap U_{\infty})\simeq U_{0}\cap U_{\infty}\simeq{\delta'}_{0}^{-1}(U_{0}\cap U_{\infty})$
of two algebraic spaces $\delta'_{\infty}:\mathfrak{C}'_{\infty}\rightarrow U_{\infty}$
and $\delta'_{0}:\mathfrak{C}'_{0}\rightarrow U_{0}$ which are constructed
as follows:

- The algebraic space $\mathfrak{C}'_{\infty}$ is obtained from $U_{\infty}$
be replacing the point $p_{\infty}$ by four copies of itself, one
for each irreducible component in the reduced degenerate fiber ${\pi'}_{R}^{-1}(p_{\infty})$
(resp. ${\pi'}_{F}^{-1}(p_{\infty})$).

- Corresponding to the fact that the degenerate fiber ${\pi'}_{R}^{-1}(p_{0})$
(resp. ${\pi'}_{F}^{-1}(p_{0})$) has two irreducible components,
both occurring with multiplicity two, the algebraic space $\mathfrak{C}'_{0}$
is now itself a compound object. First we let $g:\tilde{U}_{0}\rightarrow U_{0}$
be Galois cover of degree two ramified at $p_{0}$ and \'etale elsewhere.
Then we let $\tilde{\mathfrak{D}}'_{0}\rightarrow\tilde{U}_{0}$ be
the scheme obtained by gluing two copies $\tilde{U}_{0,\pm}$ of $\tilde{U}_{0}$
by the identity outside $\tilde{p}_{0}=g^{-1}(p_{0})$. The Galois
group $\mu_{2}$ acts freely on $\tilde{\mathfrak{D}}'_{0}$ by $\tilde{U}_{0,\pm}\ni\tilde{p}\mapsto-\tilde{p}\in\tilde{U}_{0,\mp}$
and we let $\gamma_{0}':\mathfrak{D}'_{0}=\tilde{\mathfrak{D}}'_{0}/\mu_{2}\rightarrow U_{0}\simeq\tilde{U}_{0}/\mu_{2}$
be the geometric quotient taken in the category of algebraic spaces.
Finally, $\delta'_{0}:\mathfrak{C}'_{0}\rightarrow U_{0}$ is obtained
by gluing two copies $\gamma'_{0,i}:\mathfrak{D}'_{0,i}\rightarrow U_{0}$,
$i=1,2$ of $\mathfrak{D}'_{0}$ by the identity along the open subschemes
${\gamma'_{0,1}}^{-1}(U_{0}\setminus\{p_{0}\})\simeq U_{0}\setminus\{p_{0}\}\simeq{\gamma'_{0,2}}^{-1}(U_{0}\setminus\{p_{0}\})$. 
\end{rem}
\bibliographystyle{amsplain}

\providecommand{\bysame}{\leavevmode\hbox to3em{\hrulefill}\thinspace}
\providecommand{\MR}{\relax\ifhmode\unskip\space\fi MR }
\providecommand{\MRhref}[2]{%
  \href{http://www.ams.org/mathscinet-getitem?mr=#1}{#2}
}
\providecommand{\href}[2]{#2}
\begin{thebibliography}{}

\end{thebibliography}


\begin{thebibliography}{99} 

\bibitem{AFKKZ12} I. Arzhantsev, H. Flenner, S. Kaliman, F. Kutzschebauch and  M. Zaidenberg, \emph{Flexible varieties and automorphism groups}, To appear in Duke Math. J., preprint arXiv:1011.5375.

\bibitem{AKZ12} I. V. Arzhantsev, K. Kuyumzhiyan and M. Zaidenberg, \emph{Flag varieties, toric varieties, and suspensions: three instances of infinite transitivity}, Mat. Sb. 203 7 (2012), 3-30.

\bibitem{BML08} T.  Bandman and L. Makar-Limanov, \emph{On $\mathbb{C}$-fibrations over projective curves}, Michigan Math. J. 56 (2008), 669-686. 

\bibitem{CML05} A. Crachiola and L. Makar-Limanov, \emph{On the rigidity of small domains}, J. Algebra 284 (2005), 1-12.

\bibitem{CraThesis} A. Crachiola, \emph{On the AK-Invariant of Certain Domains}, Ph.D. thesis, Wayne State University, Detroit, Michigan, 2004.

\bibitem{Dan89} W. Danielewski, \emph{On the cancellation problem and automorphism groups of affine algebraic varieties}, Preprint, Warsaw, 1989. 

\bibitem{DolgBook} I. Dolgachev, \emph{Lectures on invariant theory}, LMS Lecture Notes Series, 296, CUP, 2003.

\bibitem{Dub04} A. Dubouloz, \emph{Completions of normal affine surfaces with a trivial Makar-Limanov invariant},  Michigan Math. J. 52 (2004), 289-308.

\bibitem{Dub11} A. Dubouloz, \emph{Complements of hyperplane sub-bundles in projective space bundles over $\mathbb{P}^1$}, preprint arXiv:1108.6209.

\bibitem{DubKish12} A. Dubouloz and T. Kishimoto, \emph{Log-uniruled affine varieties without cylinder-like open subsets}, preprint arXiv:1212.0521.

\bibitem{Fie94} K.-H. Fieseler, \emph{On complex affine surfaces with $\mathbb{C}^+$-action}, Comment. Math. Helv. 69 (1994), 5-27.

\bibitem{FreuBook} G. Freudenburg, \emph{Algebraic Theory of Locally Nilpotent Derivations}, Encycl. Math. Sci., 136, Inv. Theory and Alg. Tr. Groups, VII, Springer-Verlag, 2006.

\bibitem{Giz71} M. H. Gizatullin: I. \emph{Affine surfaces that are quasihomogeneous with respect to an algebraic group},  Math. USSR Izv. 5 (1971), 754-769; II. \emph{Quasihomogeneous affine surfaces} ibid. 1057-1081.

\bibitem{MLNotes} L. Makar-Limanov, \emph{Locally nilpotent derivations, a new ring invariant and applications}, Lecture notes, Bar-Ilan University, 1998. Avail. at at http://www.math.wayne.edu/~lml/

\bibitem{Mur69} M. Pavaman Murthy, \emph{Vector Bundles Over Affine Surfaces Birationally Equivalent to a Ruled Surface}, Annals of Maths. vol. 89, No. 2 (1969),  242-253.  
\end{thebibliography}

\end{document}